\documentclass[12pt]{amsart}
\usepackage{amssymb,amsmath,amsthm}
\usepackage{graphicx}
\usepackage{float}
\numberwithin{equation}{section}
\allowdisplaybreaks
\newtheorem{thm}{Theorem}[section]
\newtheorem{lemma}[thm]{Lemma}
\newtheorem{remark}[thm]{Remark}

\begin{document}
\title[LIMIT OF EXTREME EIGENVALUES OF RANDOM QUATERNION MATRICES]{ON THE  LIMIT OF EXTREME EIGENVALUES OF LARGE DIMENSIONAL RANDOM QUATERNION MATRICES}

\author{Yanqing Yin,\ \ Zhidong Bai, \ \ Jiang Hu}
\thanks{Z. D. Bai was partially supported by CNSF  11171057, PCSIRT, and Fundamental Research Funds for the Central Universities; J. Hu was partially supported by a grant CNSF 11301063.
}

\address{KLASMOE and School of Mathematics \& Statistics, Northeast Normal University, Changchun, P.R.C., 130024.}
\email{yinyq799@nenu.edu.cn}
\address{KLASMOE and School of Mathematics \& Statistics, Northeast Normal University, Changchun, P.R.C., 130024.}
\email{baizd@nenu.edu.cn}
\address{KLASMOE and School of Mathematics \& Statistics, Northeast Normal University, Changchun, P.R.C., 130024.}
\email{huj156@nenu.edu.cn}


\maketitle
\begin{abstract}
Since E.P.Wigner (1958) established his famous semicircle law,
lots of attention has been paid by physicists, probabilists and statisticians to study the asymptotic properties of the largest eigenvalues for random matrices.
Bai and Yin (1988) obtained  the necessary and sufficient conditions for the strong convergence of the extreme eigenvalues of a Wigner matrix.
In this paper, we consider the case of quaternion self-dual Hermitian matrices.
We prove the necessary and sufficient conditions for the strong convergence of extreme eigenvalues of quaternion self-dual Hermitian matrices corresponding to the Wigner case.
\end{abstract}

\section{Introduction }

In nuclear physics, the energy levels are described by the eigenvalues of the Hamiltonian $\mathbf H_n$,
which is regarded as an Hermitian matrix of very large order $n$. In the absence of any precise knowledge of $\mathbf H_n$,
one assumes a reasonable probability distribution for its elements, from which we can deduce statistical properties of its spectrum.
In history, the idea of a statistical mechanics of nuclei based on an ensemble of systems is due to Wigner \cite{wigner1951statistical,wigner1955,wigner1957},
and is further developed by Dyson \cite{dyson1962threefold}, Gaudin and Mehta \cite{mehta2004random}. Those foundational works can be viewed as the cornerstone of building-up the random matrix theory (RMT). The definition of random matrix ensembles consists of two parts \cite{ginibre:440}:
the algebraic structure of the matrices $\mathbf H_n$ and the probability distributions of  $\mathbf H_n$.
In accordance to the consequences of time-reversal invariance, three kinds of ensembles were proposed: \emph{Gaussian orthogonal ensemble} (\emph{GOE}),
 \emph{Gaussian unitary ensemble} (\emph{GUE}), and \emph{Gaussian symplectic ensemble} (\emph{GSE}) (See \cite{mehta2004random} for details of those models ).
In algebraic structure, the matrices $\mathbf H_n$ will be $n\times n$ real symmetric matrix, $n\times n$ complex Hermitian matrix and $n\times n$ quaternion self-dual Hermitian matrix, respectively.
In the past several decades, Wigner's program has drawn considerable attention from a large number of researchers and a huge body of literatures is aiming to get better understanding on those models.
 Preparing for what follows in this paper, it is necessary here to provide a short introduction to the quaternion self-dual Hermitian matrix since its structure is not as clear as the first two classes.

Quaternions were invented in 1843 by the Irish mathematician William Rowen Hamilton after a lengthy struggle to extend the theory of complex numbers to three dimensions \cite{hamilton1866elements,1973},
and it is well known that the quaternion field $\mathbb{Q}$ can be represented as a two-dimensional complex vector space \cite{chevalley1946lie}.
 Thus this representation associates any $n\times n$ quaternion matrix with a $2n\times2n$ complex matrix.
 Note that in this paper, we only use the $2n\times2n$ complex representation of the quaternion matrices.
Define four $2\times 2$ matrices:
$$\mathbf{e}=\left(
              \begin{array}{cc}
                1 & 0 \\
                0 & 1 \\
              \end{array}
            \right), \quad
\mathbf{i}=\left(
              \begin{array}{cc}
                i & 0 \\
                0 & -i \\
              \end{array}
            \right), \quad
\mathbf{j}=\left(
              \begin{array}{cc}
                0 & 1 \\
                -1 & 0 \\
              \end{array}
            \right), \quad
\mathbf{k}=\left(
              \begin{array}{cc}
                0 & i \\
                i & 0 \\
              \end{array}
            \right)
$$
where $i=\sqrt{-1}$.
We can verify that:
 \begin{gather*}
\mathbf{i} ^2=\mathbf{j} ^2=\mathbf{k} ^2=-\mathbf{e},\quad \mathbf{i} =\mathbf{j} \mathbf{k} =-\mathbf{k} \mathbf{j} ,\\
 \mathbf{j} =\mathbf{i} \mathbf{k} =-\mathbf{k} \mathbf{i} ,\quad \mathbf{k} =\mathbf{i} \mathbf{j} =-\mathbf{j} \mathbf{i} .
\end{gather*}
For any real numbers $a,b,c,d,$ let
$$q=a\mathbf{e}+b\mathbf{i}+c\mathbf{j}+d\mathbf{k}=\left(
                                                      \begin{array}{cc}
                                                        a+bi & c+di \\
                                                        -c+di & a-bi \\
                                                      \end{array}
                                                    \right)=\left(
                                                              \begin{array}{cc}
                                                                \lambda & \omega \\
                                                             -\bar{\omega} & \bar{\lambda} \\
                                                              \end{array}
                                                            \right),
$$
then $q$ is a quaternion. The quaternion conjugate of $q$ is defined by
$$\bar{q}=a\mathbf{e}-b\mathbf{i}-c\mathbf{j}-d\mathbf{k}=\left(
                                                      \begin{array}{cc}
                                                        a-bi & -c-di \\
                                                        c-di & a+bi \\
                                                      \end{array}
                                                    \right)=\left(
                                                              \begin{array}{cc}
                                                                \bar{\lambda} & -\omega \\
                                                             \bar{\omega} & \lambda \\
                                                              \end{array}
                                                            \right),
$$
and the quaternion norm of $q$ is defined by
$$\| q \| = \sqrt {{a^2} + {b^2} + {c^2} + {d^2}}  = \sqrt {{{\left| \lambda  \right|}^2} + {{\left| \omega  \right|}^2}}. $$
\begin{remark}\label{re:1}
We note that when we use the $2\times2$ complex representation of quaternion $q$,
the quaternion conjugate is in fact the ordinary conjugate transpose and $\| q \|^2$ is equal to the determinant of $q$.
\end{remark}
An $n\times n$ quaternion self-dual Hermitian matrix $\mathbf{H_n}=(x_{jk})$ is a matrix whose entries $x_{jk} \ (j,k=1,\cdots,n)$ are quaternions and satisfy $x_{jk}=\bar{x}_{kj}$.
Using the $2n\times2n$ complex representation of $\mathbf H_n$ and according to Remark \ref{re:1}, $\mathbf H_n$ is obviously a $2n\times2n$ Hermitian matrix.
Then, we represent the entries of $\mathbf H_n$ as
$$x_{jk} = \left( {\begin{array}{*{20}{c}}
a_{jk} + b_{jk}i & c_{jk} + d_{jk}i \\
-{ c }_{jk}+{ d }_{jk}i&{ a }_{jk}-{ b }_{jk}i
\end{array}} \right)
=\left( {\begin{array}{*{20}{c}}
\lambda_{jk} &\omega_{jk} \\
{ - \overline \omega }_{jk}&{\overline \lambda }_{jk}
\end{array}} \right), 1\le j<k\le n, $$
 and
$x_{jj} = \left( {\begin{array}{*{20}{c}}
a_{jj} & 0 \\
0& a_{jj}
\end{array}} \right),$ where $ a_{jk},b_{jk},c_{jk},d_{jk} \in \mathbb{R} \ {\rm and} \ 1\le j,k \le n.$  Thus we can study the eigenvalues in the ordinary sense \cite{ginibre:440}.
\begin{remark}\label{a2}
It is known (see \cite{zhang1997quaternions}) that the multiplicities of all the eigenvalues (obviously they are all real) of $\mathbf H_n$ are even and at least 2.
\end{remark}
In RMT, suppose an $n\times n$ matrix $\mathbf W_n$ is Hermitian and let real numbers $s_1\geq s_2\geq\cdots\geq s_n$ denote the eigenvalues,
then the empirical spectral distribution (ESD) of $\mathbf W_n$ is defined as: $${F^{\mathbf{W_n}}}(x) = \frac{1}{n}\sum\limits_{i = 1}^n {I({s _i} \le x)},$$
where ${I(\cdot})$ is the indicator function. It is well known that if the entries on and above the diagonal of $\mathbf W_n$ are independent random variables with zero-mean and variance $\sigma^2$,
 then ${F^{\frac{1}{\sqrt n}{\mathbf W_n}}}(x)$ converges almost surely (a.s.) to a non-random distribution $F$ with density function:
\begin{align}
    f(x)=\frac{1}{2\pi\sigma^2}\sqrt{4\sigma^2-x^2},~~x\in[-2\sigma,2\sigma],\label{01}
\end{align}
which is  known as the semicircle law. The semicircle law plays a very important role in many fields. However,
one can not get the description of the behavior of the largest and the smallest eigenvalues from this distribution function.
Juh\'{a}sz \cite{juhasz1978spectrum} and F\"{u}redi and Koml\'{o}s \cite{Furedi1981} firstly studied the asymptotic properties of the extreme eigenvalues for Wigner matrix.
 In 1988, Bai and Yin \cite{BaiYin1993} got the necessary and sufficient conditions for the almost surely convergence of the extreme eigenvalues.
Then in \cite{tracy1994level}, Tracy and Widom derived the limiting distribution of the largest eigenvalue of \emph{GOE},
 \emph{GUE}, and \emph{GSE}. In \cite{yinbai2013}, the semicircular law for quaternion self-dual Hermitian matrices was proved under the most general condition of finite second moment. More recent results can be found in \cite{o2013universality,tao2011wigner,Wang2012,PhysRevE.77.041108,Dumitriu2008} and references therein.

\section{Main theorem }

In this paper, we study the asymptotic properties of the extreme eigenvalues for quaternion self-dual Hermitian matrices without Gaussian assumption.
We shall show that the extreme eigenvalues of a high dimensional quaternion self-dual Hermitian matrix remain in certain bounded intervals. The main theorem can be described as following.
\begin{thm}\label{th1}
Suppose ${\mathbf H_n}=(x_{jk})$,
where $x_{jk}=a_{jk}\mathbf{e}+b_{jk}\mathbf{i}+c_{jk}\mathbf{j}+d_{jk}\mathbf{k}$
is a quaternion self-dual Hermitian matrix whose entries on and above the diagonal are iid quaternion random variables.
Then the largest eigenvalue of ${\mathbf Q_n}=\frac{1}{\sqrt n}{\mathbf H_n}$ tends to $\xi$ almost surely if and only if the following conditions  hold:
\begin{equation}\label{con}
\begin{split}
({\rm i})\quad &{\rm E}\left\|x_{11}^+\right\|^2<\infty,\quad where \ x^+={\rm max}(x,0), \\
({\rm ii})\quad &{\rm E}a_{12} \leq 0 \quad and \quad {\rm E}b_{12}={\rm E}c_{12}={\rm E}d_{12}=0, \\
({\rm iii})\quad &{\rm E}\left\|x_{12}-{\rm E}x_{12}\right\|^2=\sigma^2, \mbox{ and } \xi=2\sigma,\\
({\rm iv})\quad &{\rm E}\left\|x_{12}\right\|^4<\infty.\\
\end{split}
\end{equation}
\end{thm}
According to the symmetry of the largest and smallest eigenvalues of a Hermitian matrix,
we can easily derive the necessary and sufficient conditions for the existence of the limit of the smallest eigenvalue of a quaternion self-dual Hermitian matrix. Then we have the following theorem.
\begin{thm}\label{th2}
Suppose ${\mathbf H_n}=(x_{jk})$, where
$x_{jk}=a_{jk}\mathbf{e}+b_{jk}\mathbf{i}+c_{jk}\mathbf{j}+d_{jk}\mathbf{k}$
is a quaternion self-dual Hermitian matrix whose entries on and above the diagonal are iid quaternion random variables.
Then the largest eigenvalue of ${\mathbf Q_n}=\frac{1}{\sqrt n}{\mathbf H_n}$ tend to $\xi_1$ and the smallest eigenvalue of ${\mathbf Q_n}$ tends to $\xi_2$ almost surely if and only if the following conditions  hold:
\begin{equation}\label{con2}
\begin{split}
({\rm i})\quad &{\rm E}\left\|x_{11}\right\|^2={\rm E}(a_{11})^2<\infty, \\
({\rm ii})\quad &{\rm E}a_{12}={\rm E}b_{12}={\rm E}c_{12}={\rm E}d_{12}=0, \\
({\rm iii})\quad &{\rm E}\left\|x_{12}\right\|^2=\sigma^2, \mbox{ and } \xi_1=2\sigma,\ \xi_2=-2\sigma,\\
({\rm iv})\quad &{\rm E}\left\|x_{12}\right\|^4<\infty.\\
\end{split}
\end{equation}
\end{thm}

\section{Sufficiency of Conditions of Theorem \ref{th1}}

Now we are in position to present the proof of the sufficiency of conditions in (\ref{con}). Obviously, we can assume $\sigma=1$ without loss of generality.
To begin with, we shall give some lemmas that will be used. We note that $s_{\max}(\mathbf Q_n)$ and $s_{\min}(\mathbf Q_n)$ denote the largest and the smallest eigenvalues of matrix $\mathbf Q_n$ respectively throughout this paper.
\subsection{Some knowledge in graph theory}
To begin with, we shall introduce some acknowledge of graph theory. A graph is a triple $(E,V,F)$, where $E$ is the set of edges, $V$ is the set of vertices and $F$ is a function, $F:E\mapsto V\times V$. If $F(e)=(v_1,v_2)$, the vertices $v_1,v_2$ are called the ends of edge $e$, $v_1$ is the initial of $e$ and $v_2$ is terminal of $e$. If $v_1=v_2$, edge $e$ is called a loop. If two edges have the same set of ends, they are said to be coincident.

Let ${\mathbf L}=(l_1,\cdots,l_k)$ be a vector valued on $\{1,\cdots,n\}^k$. We define a $\Gamma$-graph as follows. Draw a horizontal line and plot the numbers $l_1,\cdots,l_k$ on it. Consider the distinct numbers as vertices, and draw k edges $e_j$ from $l_j$ to $l_{j+1}$, $j=1,\cdots,k$, where $l_{k+1}=l_1$. Denote the number of distinct $l_j$'s by t. Such a graph is called a $\Gamma(k,t)$-graph. An example of $\Gamma$(8,5)-graph is shown in Figure 1.

\begin{figure}[H]
\centering
\includegraphics[width=0.8\textwidth]{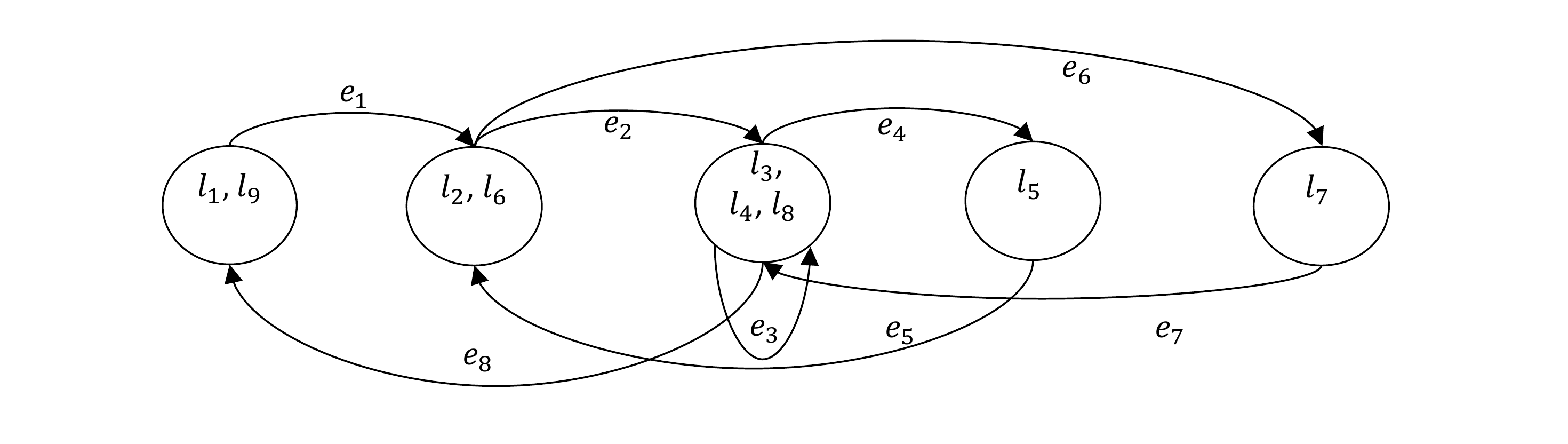}
\caption{A $\Gamma$-graph}
\end{figure}

Two $\Gamma(k,t)$-graphs are said to be isomorphic if one can be converted to the other by a permutation of $(1,\cdots,n)$. By the definition, all $\Gamma(k,t)$-graphs are classified into isomorphism classes.

A $\Gamma(k,t)$-graph is called canonical if it has the following properties:
\begin{enumerate}
    \item It's vertex set is $V=\{1,\cdots,t\}$.
    \item It's edge set is $E=\{e_1,\cdots,e_k\}$.
    \item There is a function $g$ from $\{1,\cdots,k\}$ onto $\{1,\cdots,t\}$ satisfying $g(1)=1$ and $g(l)\leq \max\{g(1),\cdots,g(l-1)\}+1$ for $1<l\leq k$.
    \item $F(e_l)=(g(l),g(l+1))$, for $l=1,\cdots,k$, with convention $g(k+1)=g(1)=1$.
  \end{enumerate}
  It's easy to see that each isomorphism class contains one and only one canonical $\Gamma(k,t)$-graph that is associated with a function $g$, and a general graph in this class can be defined by $F(e_j)=(l_{g(j)},l_{g(j+1)})$. Therefore, we obtain each isomorphism class contains $n(n-1)\cdots(n-t+1)$ $\Gamma(k,t)$-graphs.
  The canonical $\Gamma(k,t)$-graphs can be classified into three categories:
  \begin{description}
    \item[Category 1 (denoted by $\Gamma_1(k,t)$)] A canonical $\Gamma(k,t)$-graph is said to belong to category 1 if each edge is coincident with exactly  one other edge with opposite direction and the graph of noncoincident edges forms a tree (i.e., a connected graph without cycle). Obviously, there is no $\Gamma_1(k,t)$ if $k$ is odd. An example of $\Gamma_1$-graph is shown in Figure 2.
        
\begin{figure}[H]
\centering
\includegraphics[width=0.8\textwidth]{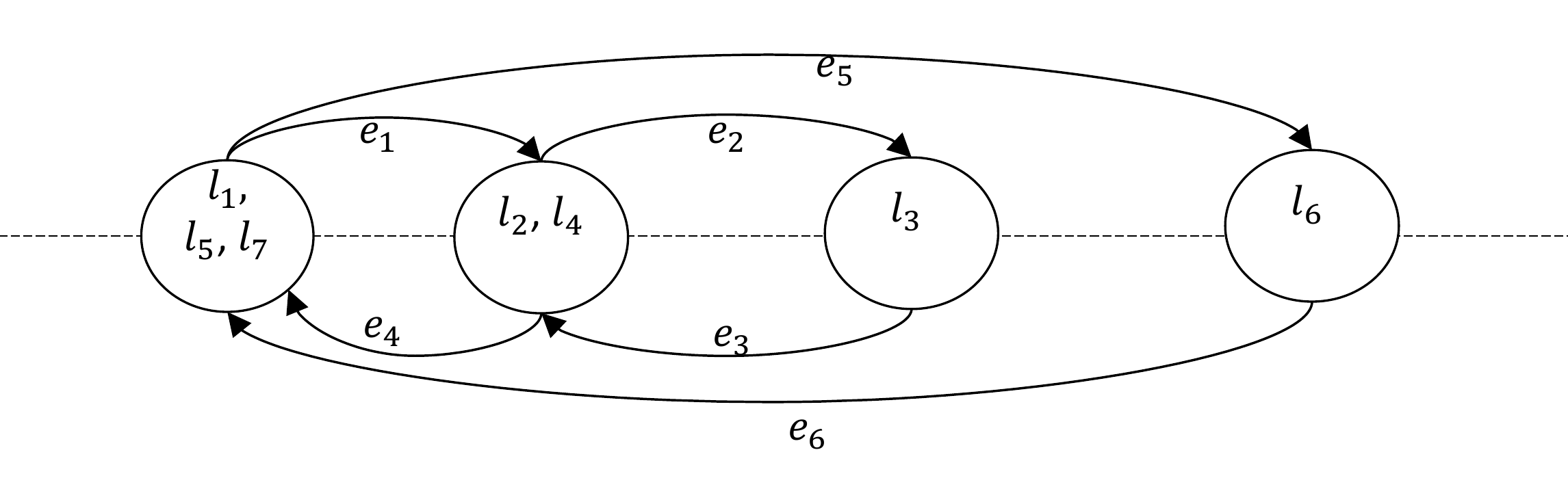}
\caption{A $\Gamma_1$-graph}
\end{figure}

    \item[Category 2 ($\Gamma_2(k,t)$)] This category consists of all those canonical $\Gamma(k,t)$-graphs that have at least one single edge (an edge not coincident with any other edges). An example of $\Gamma_2$-graph is shown in Figure 3.
\begin{figure}[H]
\centering
\includegraphics[width=0.8\textwidth]{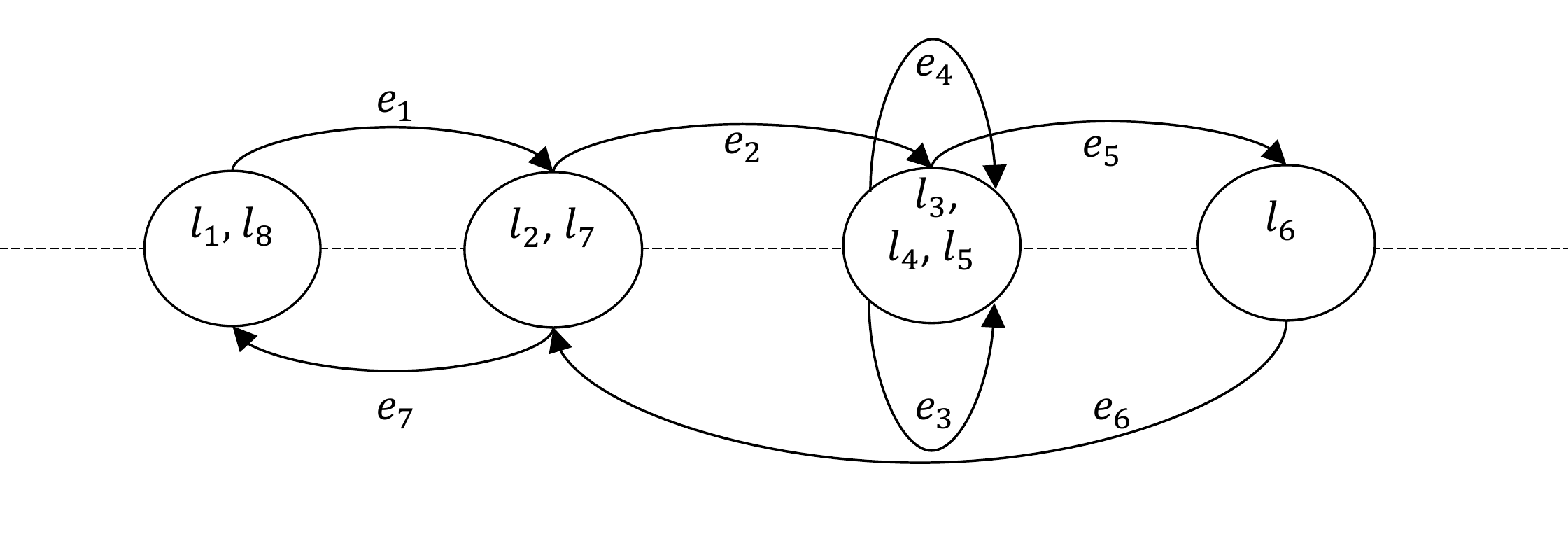}
\caption{A $\Gamma_2$-graph}
\end{figure}

    \item[Category 3 ($\Gamma_3(k,t)$)] This category consists of all other canonical $\Gamma(k,t)$-graphs. Two examples of $\Gamma_3$-graph are shown in Figure 4 (one noncoincident edge has multiplicity 4) and Figure 5 (non-coincident edges form a cycle).
\begin{figure}[H]
\centering
\includegraphics[width=0.8\textwidth]{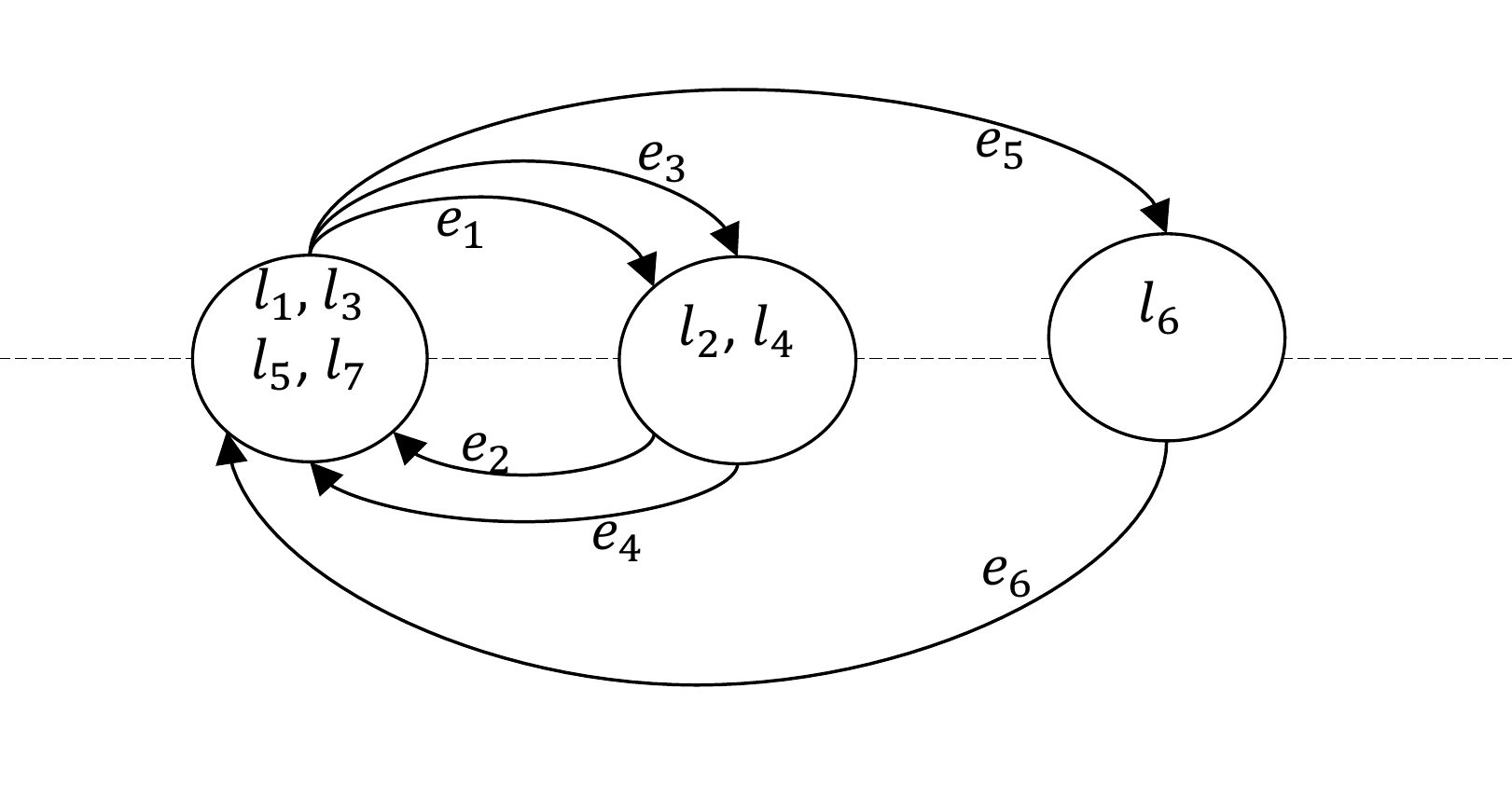}
\caption{A $\Gamma_3$-graph}
\end{figure}  
\begin{figure}[H]
\centering
\includegraphics[width=0.8\textwidth]{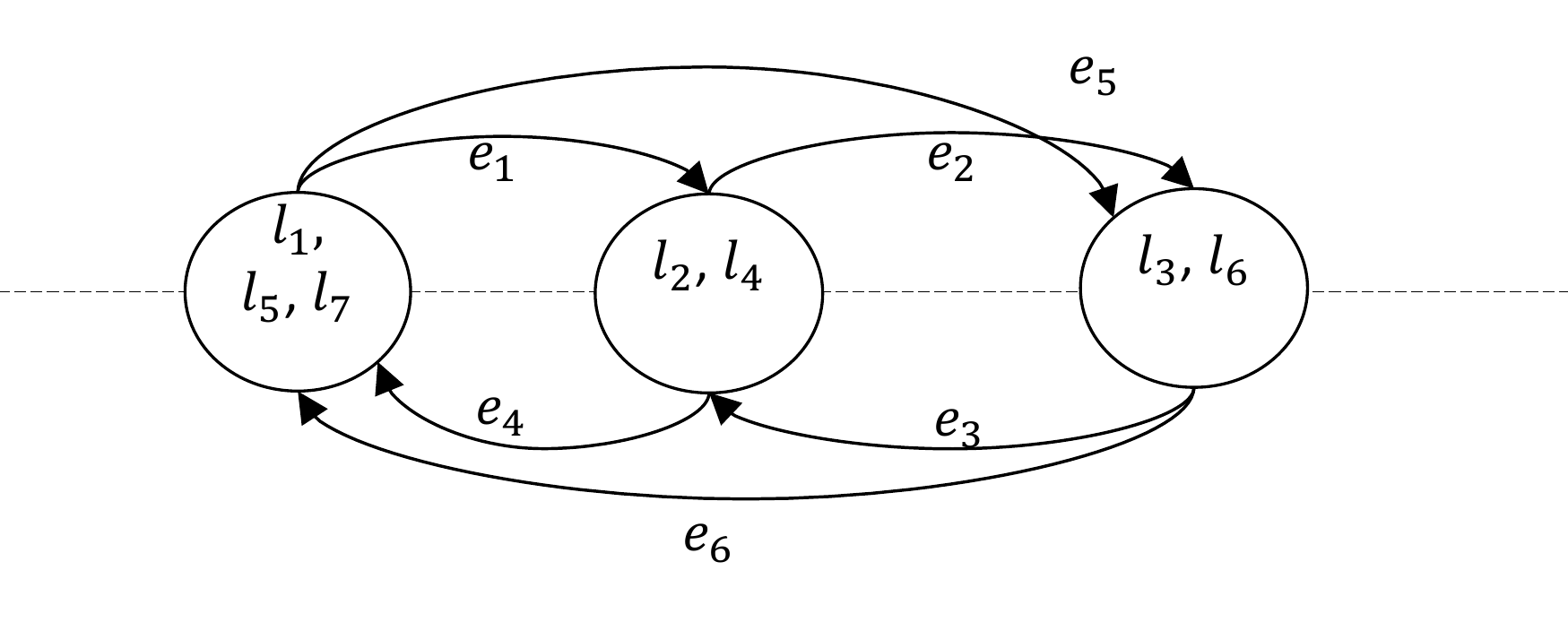}
\caption{A $\Gamma_3$-graph}
\end{figure}
  \end{description}
  Obviously, in a $\Gamma_3(k,t)$-graph, we have $t\leq(k+1)/2$.

  Now, classify the edges of a canonical $\Gamma(k,t)$-graph into several types:
  \begin{enumerate}
    \item If $g(a+1)={\rm max}(g(1),\cdots,g(a))+1$, the edge $e_a=(g(a),g(a+1))$ is called an innovation or a Type-1 $(T_1)$ edge. A $T_1$ edge leads to a new vertex in the path $e_1,\cdots,e_a$.
    \item An edge is called a $T_3$ edge if it coincides with an innovation that is single until the $T_3$ edge appears. A $T_3$ edge $(g(a),g(a+1))$ is said to be irregular if there is only one innovation single up to $a$ (an edge $e$ is said to be single up to $a$ if it doesn't coincide with any other edges in $\{(g(1),g(2)),\cdots,(g(a-1),g(a))\}$). All other $T_3$ edges are called regular $T_3$ edges.
    \item All other edges are called $T_4$ edges.
    \item The first appearance of a  $T_4$ edge is called a $T_2$ edge. There are two cases: the first is the first appearance of a single noninnovation, and the second is that the first appearance of an edge that coincides with a $T_3$ edges.
  \end{enumerate}
  Then, we shall give some lemmas without the proof.
  \begin{lemma}[Lemma 5.5 in \cite{bai2010spectral}]\label{tu1}
  Let $t$ denote the number of $T_2$ edges and $l$ denote the number of innovations in  $\{(g(1),g(2)),\cdots,(g(a-1),g(a))\}$ that are single up to $a$ and have a vertex coincident with $g(a)$. Then $l\leq t+1$.
  \end{lemma}
  \begin{lemma}[Lemma 5.6 in \cite{bai2010spectral}]\label{tu2}
  The number of regular $T_3$ edges is not greater than twice the number of $T_2$ edges.
  \end{lemma}
\subsection{Some auxiliary lemmas}

\begin{lemma}\label{lem1}
If condition {\rm(i)} of Theorem \ref{th1} holds, we have $$\limsup \frac{1}{\sqrt n}\max_{k\le n}a_{kk}^+=0.$$
\end{lemma}
\begin{proof} By Borel-Cantelli Lemma, it follows that
\begin{align*}
{\rm E}(a_{11}^+)^2<\infty
\Rightarrow &\sum_{n}{\rm P}\left((a_{11}^+)^2>\varepsilon n\right) <\infty \Rightarrow \frac{(a_{nn}^+)^2}{n}\to 0 ,\mbox{a.s.}\end{align*}
Thus, $\frac{1}{\sqrt n}\max\limits_{k\le n}a_{kk}^+\to 0,\mbox{a.s.}$.
\end{proof}

\begin{lemma}\label{lem2}
Denote $\widetilde {\mathbf Q}_n=\frac{1}{\sqrt n}\left(\widetilde x_{jk}\right)$, where
\begin{align*}
\widetilde x_{jk}=\left\{\begin{array}{cc}
\left(
                      \begin{array}{cc}
                        0 & 0 \\
                        0 & 0 \\
                      \end{array}
                    \right)
&j=k\\
x_{jk}-{\rm E}x_{jk}&j\neq k
\end{array}\right..
\end{align*}
Then if the conditions  (\ref{con}) hold, then we have $\limsup s_{\max}\left(\mathbf Q_n\right)\le \limsup s_{\max}(\widetilde{\mathbf Q}_n)$.
\end{lemma}

\begin{proof} By condition ({\rm ii}) and let $\mu={\rm E}a_{12}$, applying Lemma \ref{lem1}, one has
\begin{align*}
&s_{\max}\left(\mathbf Q_n\right)=\max_{\left\|\mathbf z\right\|_{E}=1}\mathbf z^*\mathbf Q_n \mathbf z=\frac{1}{\sqrt n}\max_{\left\|\mathbf z\right\|_{E}=1}\left(\sum_{jk}z_j^*x_{jk}z_k\right)\\
=&\max_{\left\|\mathbf z\right\|_{E}=1}\left[\frac{1}{\sqrt n}\sum_{j\neq k}z_j^*\left(x_{jk}-{\rm E}x_{jk}\right)z_k+\frac{\mu}{\sqrt n}\sum_{j,k}z_j^*z_k+\frac{1}{\sqrt n}\sum_j\left(a_{jj}-\mu\right)\left\|z_j\right\|_E^2\right]\\
\le&\max_{\left\|\mathbf z\right\|_{E}=1}\left[\frac{1}{\sqrt n}\sum_{j\neq k}z_j^*\left(x_{jk}-{\rm E}x_{jk}\right)z_k+\frac{1}{\sqrt n}\sum_j\left(a_{jj}^+-\mu\right)\left\|z_j\right\|_E^2\right]\\
\le&\max_{\left\|\mathbf z\right\|_{E}=1}\left[\frac{1}{\sqrt n}\sum_{j\neq k}z_j^*\left(x_{jk}-{\rm E}x_{jk}\right)z_k+\frac{1}{\sqrt n}\max_j\left(a_{jj}^+-\mu\right)\right]\\
\le&s_{\max}\left(\widetilde{\mathbf Q}_n\right)+o_{\rm a.s.}(1)
\end{align*}
where \ $\mathbf z=\left(z_{1}^{\prime},z_{2}^{\prime},\cdots,z_{n}^{\prime}\right)^{\prime}, z_j$ is a $2 \times 1$ vector for $j=1,\cdots,n$ and  $\left\|\cdot\right\|_E$ denotes the Euclidean norm.
\end{proof}

\begin{lemma}\label{lem3}
If condition {\rm(iv)} of Theorem \ref{th1} holds, we can select a sequence of constants \ $\eta_n \downarrow 0$ satisfying $${\rm P}(\mathbf{\widetilde{Q}_n} \neq \mathbf{\widehat{Q}_n}, \ i.o.)=0,$$
where $\mathbf{\widehat{Q}_n}=\frac{1}{\sqrt n}\left( x_{jk}(1-\delta_{jk})I(\|x_{jk}\|\leq \eta_n \sqrt n)\right)$, and $\delta_{jk}$ is the Kronecker delta.
And the speed of $\eta_n \downarrow 0$ can be made arbitrarily slow.
\end{lemma}
\begin{proof}
Note that Lemma \ref{lem3} can be viewed as a generalization of Lemma 3.1 (Truncation Lemma) in \cite{bai1988necessary}, we shall omit the proof for brevity.
\end{proof}

By Lemma \ref{lem3} and the fact that for any Hermitian matrices $A$ and $B$, $s_{\max}(A+B)\leq s_{\max}(A)+s_{\max}(B)$.
Note that by the selection of $\eta_n$, we have $s_{\max}({\rm E}(\mathbf{\widehat{Q}_n}))\to 0$.
We need only investigate the upper limit of $s_{\max}(\mathbf{\widehat{Q}_n}-{\rm E}(\mathbf{\widehat{Q}_n}))$. Combining Lemma \ref{lem2} and Lemma \ref{lem3}, we shall prove the sufficiency of Theorem \ref{th1} under the following four assumptions:

\begin{equation}
\begin{split}
({\rm i})\quad &x_{jj}=\left(
                         \begin{array}{cc}
                           0 & 0 \\
                           0 & 0 \\
                         \end{array}
                       \right).\\ \notag
({\rm ii})\quad &{\rm E}a_{jk}={\rm E}b_{jk}={\rm E}c_{jk}={\rm E}d_{jk}=0,\sigma_n^2={\rm E}\left\|x_{jk}\right\|^2\le 1 \ {\rm for}\ j\neq k. \\\notag
({\rm iii})\quad &\left\|x_{jk}\right\|\le\eta_n\sqrt n \ \ {\rm for} \ j\neq k.\\\notag
({\rm iv})\quad &{\rm E}\left\|x_{jk}^l\right\|\le b\left(\eta_n\sqrt n\right)^{l-3} \ {\rm \ for \ some \ constant} \ b>0 \ {\rm \ and \ all} \ j\neq k,l\ge3.\notag
\end{split}
\end{equation}

\subsection{The proof of sufficiency of Theorem \ref{th1}}

Due to the Theorem {\bf 1.1} of \cite{yinbai2013}, we have
\begin{align*}
\liminf_{n\to \infty}s_{\max}\left(\mathbf Q_n\right)\ge 2,{\rm a.s..}
\end{align*}
Thus, it is sufficient to show that
\begin{align*}
\limsup_{n\to \infty}s_{\max}\left(\mathbf Q_n\right)\le 2,{\rm a.s..}
\end{align*}
For any even integer k and real number $\eta>2$, we have
\begin{align}\label{al:1}
{\rm P}\left(s_{\max}\left(\mathbf Q_n\right)\ge\eta\right)\le{\rm P}\left({\rm tr}\left[\left(\mathbf Q_n\right)^k\right]\ge\eta^k\right)\le\eta^{-k}{\rm E}\left({\rm tr}\left(\mathbf Q_n\right)^k\right).
\end{align}
To complete the proof, we shall select a sequence of $k_n=2m$ with the properties $k/\log \left(n\right)\to \infty$ and $k\eta_n^{1/3}/\log  n\to 0$, and show that the right-hand side of (\ref{al:1}) is summable. To this end,
we will devote to estimate
\begin{align}\label{al:4}
{\rm E}\left({\rm tr}\left(\mathbf Q_n\right)^k\right)=&n^{-k/2}\sum_{j_1,\cdots,j_k}{\rm Etr}\left(x_{j_1j_2}x_{j_2j_3}\cdots x_{j_kj_1}\right)\notag\\
=&n^{-k/2}\sum_{G}\sum_{\mathbf L}{\rm Etr}\left(x_G(\mathbf L)\right),
\end{align}
where the graphs G are \ $\Gamma(k,t)$-graphs defined in Subsection {3.1}. Note that if the graph G has a single edge, then the corresponding term is  zero. So we only need to consider $\Gamma_1$ and $\Gamma_3$-graphs.
Adopting the definition of types of edges in Subsection {3.1}, (\ref{al:4}) can be written as
\begin{align}\label{al:8}
{\rm E}\left({\rm tr}\left(\mathbf Q_n\right)^k\right)=n^{-k/2}{\sum}^{\prime}{\sum}^{\prime\prime}{\sum}^{\prime\prime\prime}{\rm Etr}\left(x_{j_1j_2}x_{j_2j_3}\cdots x_{j_kj_1}\right)
\end{align}
where ${\sum}^{\prime}$ is the summation for different arrangement of $T_1,T_3,T_4$-types edges,
${\sum}^{\prime\prime}$ is the summation for canonical graphs ($\Gamma_1,\Gamma_3$) of the given arrangement of edges and ${\sum}^{\prime\prime\prime}$ is the summation for isomorphic graphs of the given canonical graph.
Before estimating the righthand side of (\ref{al:8}), we establish an inequality:
\begin{lemma}\label{lem4}
For $\forall \  j_1,j_2,\cdots,j_k$,
\begin{align*}
\left|{\rm tr}\left(x_{j_1j_2}x_{j_2j_3}\cdots x_{j_kj_1}\right)\right|\le 2\left\| x_{j_1j_2}\right\|\left\|x_{j_2j_3}\right\|\cdots\left\| x_{j_kj_1}\right\|.
\end{align*}
\end{lemma}
\begin{proof}
Noticing that $x_{j_1j_2}x_{j_2j_3}\cdots x_{j_kj_1}$ can be written as $\left(\begin{array}{cc}
\alpha&\beta\\
-\bar\beta&\bar \alpha
\end{array}\right)$, and according to Remark \ref{re:1}, we have
\begin{align*}
&\left|{\rm tr}\left(x_{j_1j_2}x_{j_2j_3}\cdots x_{j_kj_1}\right)\right|=\left|\alpha+\bar\alpha\right|\le2\left(\left|\alpha\right|^2+\left|\beta\right|^2\right)^{1/2}\\
=&2\left\{\det{\left(\begin{array}{cc}
\alpha&\beta\\
-\bar\beta&\bar \alpha
\end{array}\right)}\right\}^{1/2}\\
=&2\left\{\det\left(x_{j_1j_2}\right)\det\left(x_{j_2j_3}\right)\cdots\det\left( x_{j_kj_1}\right)\right\}^{1/2}\\
= &2\left\| x_{j_1j_2}\right\|\left\|x_{j_2j_3}\right\|\cdots\left\| x_{j_kj_1}\right\|.
\end{align*}
\end{proof}
Applying Lemma \ref{lem4} to (\ref{al:8}), we have
\begin{align}\label{al:9}
{\rm E}\left({\rm tr}\left(\mathbf Q_n\right)^k\right)\le 2n^{-k/2}{\sum}^{\prime}{\sum}^{\prime\prime}{\sum}^{\prime\prime\prime}{\rm E}\left\|x_{j_1j_2}\right\|\left\|x_{j_2j_3}\right\|\cdots\left\| x_{j_kj_1}\right\|.
\end{align}
For $\Gamma_1$ and $\Gamma_3$-graphs, suppose that there are $l\ (l\le m)$ innovations and $t$ $T_2$ edges in the graph G. Then there are $l$ $T_3$ edges, $k-2l$ $T_4$ edges and $l+1$ noncoincident vertices.
We obtain that ${\sum}^{\prime}\leq\sum_{l=1}^{k/2}{k\choose l}{k-l\choose l}$. By Lemmas \ref{tu1} and \ref{tu2},
we know that ${\sum}^{\prime\prime}\leq\sum_{t=0}^{k-2l}\left(t+1\right)^{2(k-2l)}
{k^2\choose t}\left(t+1\right)^{k-2l}$. Obviously, ${\sum}^{\prime\prime\prime}$ is bounded by $n^{l+1}$. Then, together with (\ref{al:9}), we have
\begin{footnotesize}
\begin{align*}
{\rm E}\left({\rm tr}\left(\mathbf Q_n\right)^k\right)\le &2n^{-k/2}\sum_{l=1}^{[k/2]}\sum_{t=0}^{k-2l}n^{l+1}\bigg(\begin{array}{cc}k\\l\end{array}\bigg)\bigg(\begin{array}{cc}k-l\\l\end{array}\bigg)\bigg(\begin{array}{cc}k^2\\t\end{array}\bigg)
\left(t+1\right)^{3(k-2l)}b^t\left(\sqrt n \eta_n\right)^{k-2l-t}\\
\le&2n\sum_{l=1}^{[k/2]}\frac{k!}{l!l!(k-2l)!}\sum_{t=0}^{k-2l}k^{2t}
\left(t+1\right)^{3(k-2l)}b^t\eta_n^{k-2l}\left(\sqrt n\eta_n\right)^{-t}\\
=&2n\sum_{l=1}^{[k/2]}\frac{k!}{l!l!(k-2l)!}\sum_{t=0}^{k-2l}
\left(t+1\right)^{3(k-2l)}\left[\frac{bk^2}{\sqrt n \eta_n}\right]^{t+1}\eta_n^{k-2l}\frac{\sqrt n \eta_n}{bk^2}\\
= &\frac{2\eta_nn^{3/2}}{bk^2}\sum_{l=1}^{[k/2]}\frac{k!}{l!l!(k-2l)!}\sum_{t=0}^{k-2l}
\left(t+1\right)^{3(k-2l)}\left[\frac{bk^2}{\sqrt n \eta_n}\right]^{t+1}\eta_n^{k-2l}.
\end{align*}
\end{footnotesize}
Due to the elementary inequality
\begin{align*}
\alpha^{-(t+1)}\left(t+1\right)^{\beta}\le\left(\frac{\beta}{\log (\alpha)}\right)^{\beta}
\end{align*}
where $\beta>0,\alpha>1$, for large $n$ we have
\begin{align}\label{al:10}
{\rm E}\left({\rm tr}\left(\mathbf Q_n\right)^k\right)\le& \frac{2\eta_nn^{3/2}}{bk^2}\sum_{l=1}^{k/2}\frac{k!}{l!l!(k-2l)!}\sum_{t=0}^{k-2l}\eta_n^{k-2l}\left(\frac{3(k-2l)}{\log (\sqrt n \eta_n/bk^{2})}\right)^{3(k-2l)}\notag\\
\le&2 n^{3/2}\sum_{l=1}^{k/2}\frac{k!}{l!l!(k-2l)!}\eta_n^{k-2l}\left(\frac{10k}{\log  n}\right)^{3(k-2l)}\notag \\
\le&2n^{3/2}\left[1+1+\left(\frac{10k\eta_n^{1/3}}{\log  n}\right)^3\right]^k\notag\\
=&2n^{3/2}\left[2+o(1)\right]^k
\end{align}
where the last inequality follows from $k\eta_n^{1/3}/\log  n\to 0$ . Finally, together with (\ref{al:1}) and (\ref{al:10}), we have
\begin{align*}
\sum_n{\rm P}\left(s_{\max}\left(\mathbf Q_n\right)\ge\eta\right)\le &2\sum_n\frac{n^{3/2}\left[2+o(1)\right]^k}{\eta^k}\\
\le &2\sum_n n^{3/2}e^{k\log{\frac{2+o(1)}{\eta}}}\\
= &\sum_n \left(2n\right)^{3/2+{\frac{k}{\log \left(2n\right)}\log{\frac{2+o(1)}{\eta}}}}<\infty
\end{align*}
where the last inequality follows from the fact that $k/\log n\to \infty$. The sufficiency is proved.

\section{Necessity of Conditions of Theorem \ref{th1}}

\subsection{Necessity of condition (i)}
Suppose that $\limsup s_{\max}\left(\mathbf Q_n\right)\le \xi ,{\rm a.s.}$. Then,
\begin{align*}
s_{\max}\left(\mathbf Q_n\right)=&\max_{\left\|\mathbf z\right\|_{E}=1}\mathbf z^*\mathbf Q_n\mathbf z
\ge \frac{1}{\sqrt n}\left(0,\cdots,0,1\right)\mathbf Q_n\left(0,\cdots,0,1\right)^{\prime}
=\frac{a_{nn}}{\sqrt n}
\end{align*}
which implies that
$\frac{a_{nn}^+}{\sqrt n}\le\max\left\{0,s_{\max}\left(\mathbf Q_n\right)\right\}.$\\
Applying Borel-Cantelli lemma, for any $\eta>\xi$, we have
\begin{align*}
\limsup\frac{a_{nn}^+}{\sqrt n}<\eta,{\rm a.s.}\Rightarrow{\rm P}\left(\frac{a_{nn}^+}{\sqrt n}\ge\eta,{\rm i.o.}\right)=0
\Rightarrow\sum_{n=1}^{\infty}{\rm P}\left(a_{11}^+\ge\eta\sqrt n\right)<\infty,
\end{align*}
which implies ${\rm E}(a_{11}^+)^2<\infty$.

\subsection{Necessity of condition (iv)}
Assume that condition (i) holds. 
Let $\mathcal N_l=\left\{j;2^l<j\le2^{l+1};\left\|x_{jj}\right\|\le2^{l/4}\right\}$, $N_l={\#}\left(\mathcal N_l\right)$ and $p={\rm P}\left(\left\|x_{11}\right\|\le2^{l/4}\right)$. When $n\in(2^{l+1},2^{l+2}]$, for $x_{jk}\neq0$ and $j,k\in \mathcal N_l$,
 construct a unit complex vector $\mathbf { z}$
by taking $z_k^*=\left(\frac{1}{\sqrt 2},0\right),z_j^*=\left(\frac{\bar\lambda_{jk}}{\sqrt 2\left\|x_{jk}\right\|},-\frac{\omega_{jk}}{\sqrt 2\left\|x_{jk}\right\|}\right)$ and $z_l^*=\left(0,0\right)\left(l\neq j,l\neq k\right)$. Then, we obtain
\begin{align*}
& s_{\max}(\mathbf Q_n)\ge\mathbf{z^*Q_nz}=\frac{1}{\sqrt n}\sum_{j,k}{ z_j^*}x_{jk}{z_k}=\frac{1}{\sqrt n}\left[\left\|x_{jk}\right\|+\frac{1}{2}\left(a_{jj}+{a_{kk}}\right)\right]\\
\Rightarrow&s_{\max}({\mathbf Q}_n)\ge2^{-l/2-1}\max_{j,k\in\mathcal N_l}\left\{\left\|x_{jk}\right\|\right\}-2^{-l/4-1}.
\end{align*}
Obviously, when $x_{jk}=\left(
                                               \begin{array}{cc}
                                                 0 & 0 \\
                                                 0 & 0 \\
                                               \end{array}
                                             \right)
$, the conclusion is still true. Moreover, we have known that
for any $\eta>\xi$,
\begin{align*}
{\rm P}\left(\max_{2^{l+1}<n\le2^{l+2}}s_{\max}\left(\mathbf Q_n\right)\ge\eta,{\rm i.o.}\right)=0.
\end{align*}
Combining the above inequality and Borel-Cantelli lemma, it follows that
\begin{align}\label{al:11}
\sum_{l=1}^{\infty}{\rm P}\left(\max_{j,k\in\mathcal N_l}\left\{\left\|x_{jk}\right\|\right\}\ge\eta2^{l/2+1}\right)<\infty.
\end{align}
Noticing that $N_l$ and $x_{jk}$ are independent, we have
\begin{align*}
&{\rm P}\left(\max_{j,k\in\mathcal N_l}\left\{\left\|x_{jk}\right\|\right\}\ge\eta2^{l/2+1}\bigg|N_l=r\right)\\
=&{\rm P}\left(\max_{1\le j<k \le r}\left\{\left\|x_{jk}\right\|\right\}\ge\eta2^{l/2+1}\right)\\
=&1-\left[1-{\rm P}\left(\left\|x_{12}\right\|\ge\eta2^{l/2+1}\right)\right]^{r(r-1)/2}.
\end{align*}
Denote $N_l\sim B\left(2^l,p\right)$, therefore, one has \begin{align}\label{al:12}
&{\rm P}\left(\max_{j,k\in\mathcal N_l}\left\{\left\|x_{jk}\right\|\right\}\ge\eta2^{l/2+1}\right)\notag\\
=&\sum_{r=0}^{2^l}{\rm P}\left(N_l=r\right){\rm P}\left(\max_{j,k\in\mathcal N_l}\left\{\left\|x_{jk}\right\|\right\}\ge\eta2^{l/2+1}\bigg|N_l=r\right)\notag\\
\ge&\sum_{r=2^{l-1}+1}^{2^l}\left(\begin{array}{cc}
2^l\\
r\end{array}\right)p^r\left(1-p\right)^{2^l-r}
\left\{1-\left[1-{\rm P}\left(\left\|x_{12}\right\|\ge\eta2^{l/2+1}\right)\right]^{r(r-1)/2}\right\}\notag\\
\ge&\frac{1}{2}\left\{1-\left[1-{\rm P}\left(\left\|x_{12}\right\|\ge\eta2^{l/2+1}\right)\right]^{2l-3}\right\}
\end{align}
where the last inequality follows from that $p$ is close to 1 for all large $l$. From (\ref{al:11}) and (\ref{al:12}), we acquire that
\begin{align*}
&\sum_{l=1}^{\infty}\left\{1-\left[1-{\rm P}\left(\left\|x_{12}\right\|\ge\eta2^{l/2+1}\right)\right]^{2l-3}\right\}<\infty\\
\Rightarrow&\prod_{l=1}^{\infty}\left[1-{\rm P}\left(\left\|x_{12}\right\|\ge\eta2^{l/2+1}\right)\right]^{2l-3}>0\\
\Rightarrow&\sum_{l=1}^{\infty}2^{2l-3}{\rm P}\left(\left\|x_{12}\right\|\ge\eta2^{l/2+1}\right)<\infty\\
\Rightarrow&{\rm E}\left\|x_{12}\right\|^4<\infty.
\end{align*}
\subsection{Necessity of condition (ii)} Assume that conditions (i) and (iv) hold. 
Firstly, we will show that $a={\rm E}\left(a_{12}\right)\le0$. Suppose that $a>0$. Let $\mathcal D_n=\left\{j\le n,\left\|x_{jj}\right\|<n^{1/4}\right\}$, $N=\#\left(\mathcal D_n\right)$ and
$\widetilde {\mathbf Q}_n=\frac{1}{\sqrt n}\left(\widetilde x_{jk}\right)$ with $\widetilde x_{jk}=\left\{\begin{array}{cc}
\left(
                       \begin{array}{cc}
                         0 & 0 \\
                         0 & 0 \\
                       \end{array}
                     \right)
& j=k\\
x_{jk}&j\neq k\end{array}\right.$. Construct a unit vector $\mathbf z^*=\left(z_1^*,\cdots,z_n^*\right)$ by taking $z_j^*=(\frac{1}{\sqrt N},0) \ {\rm if} \ j\in\mathcal D_n \ {\rm and} \ z_j^*=(0,0) \ {\rm otherwise}$,
then we have
\begin{align*}
s_{\max}(\mathbf Q_n)\ge &\mathbf {z^*Q_nz}\\
=&\mathbf {z^*(Q_n-\widetilde Q_n)z}+\mathbf {z^*(\widetilde Q_n-{\rm E}\widetilde Q_n)z}+\mathbf {z^*({\rm E}\widetilde Q_n)z}\\
\ge&\frac{1}{N\sqrt n}\sum_{j\in \mathcal D_n}a_{jj}+\mathbf {z^*(\widetilde Q_n-{\rm E}\widetilde Q_n)z}+\frac{a(N-1)}{\sqrt n}\\
\ge &-n^{-1/4}+s_{\min}\mathbf {(\widetilde Q_n-{\rm E}\widetilde Q_n)}+\frac{a(N-1)}{\sqrt n}\\
\ge &-n^{-1/4}-2\sigma+\frac{a(N-1)}{\sqrt n}\to\infty
\end{align*}
which contradicts with the assumption that $\limsup s_{\max}\left(\mathbf Q_n\right)=2\sigma$ almost surely. Here, we have used a fact that $s_{\min}\mathbf {(\widetilde Q_n-{\rm E}\widetilde Q_n)}\to-2\sigma$, a.s. which is an easy consequence of the sufficient part of the theorem.

Now, we proceed to show that ${\rm E}b_{12}={\rm E}c_{12}={\rm E}d_{12}=0$. To this end, we shall quote the following lemma:
\begin{lemma}[Lemma 2.7 in \cite{bai2010spectral}]\label{le:5}
Let $\mathbf A_n$ be an $n\times n$ skew-symmetric matrix whose elements above the diagonal are 1 and those below the diagonal are -1. Then, the eigenvalues of $\mathbf A_n$ are $\lambda_k=-i\cot\left(\pi\left(2k-1\right)/2n\right),k=1,2,\cdots,n$.
The eigenvector associated with $\lambda_k$ is $\mathbf u_k=\frac{1}{\sqrt n}\left(1,\rho_k,\cdots,\rho_k^{n-1}\right)^{\prime}$, where $\rho_k=\left(\lambda_k-1\right)/\left(\lambda_k+1\right)=\exp\left(-i\pi\left(2k-1\right)/n\right)$.
\end{lemma}
Since we shall use matrix similarity transformation to get a similar matrix of $\mathbf Q_n$, written as:
\begin{align*}
\widetilde {\widetilde {\mathbf Q}}_n=\frac{1}{\sqrt n}\left(\begin{array}{cc}
\Sigma_1&\Sigma_2\\
-\overline{\Sigma}_2&\overline{\Sigma}_1\end{array}
\right),
\end{align*}
where $\Sigma_1=\left(\lambda_{jk}\right)_{n\times n},\Sigma_2=\left(\omega_{jk}\right)_{n\times n}$, $\bar{\Sigma}_1=\left( \bar\lambda_{jk}\right)_{n\times n}$ while $\bar{\Sigma}_2=\left( \bar{\omega}_{jk}\right)_{n\times n}$.
Of course, $\widetilde {\widetilde {\mathbf Q}}_n$ has the same eigenvalues as ${\mathbf Q}_n$. Let $\mathbf A_n$ be an $n\times n$ skew-symmetric matrix whose elements above the diagonal are 1 and those below the diagonal are -1 and
let $\breve{\mathbf Q}_n$ be the matrix obtained from $\widetilde {\widetilde {\mathbf Q}}_n$ by replacing $\Sigma_1,\Sigma_2$'s diagonal elements with zero. Let $\mathbf J$ be the $n\times n$ matrix of $1$'s and $\mathbf I$ be the $n\times n$ identity matrix. Suppose ${\rm E}b_{12}=b,{\rm E}c_{12}=c,{\rm E}d_{12}=d$,
then write
\begin{small}
\begin{align*}
&{\rm E}\breve {\mathbf Q}_n\\
=&\frac{1}{\sqrt n}\left(a\left(\begin{array}{cc}
\mathbf J-\mathbf I&0\\
0&\mathbf J-\mathbf I\end{array}
\right)+ib\left(\begin{array}{cc}
\mathbf A_n&0\\
0&-\mathbf A_n\end{array}
\right)+c\left(\begin{array}{cc}
0&\mathbf A_n\\
-\mathbf A_n&0\end{array}
\right)+id\left(\begin{array}{cc}
0&\mathbf A_n\\
\mathbf A_n&0\end{array}
\right)\right)\\
=&{\rm \mathbf R_n}+{\rm \mathbf M_n},
\end{align*}
\end{small}
where \begin{align*}
{\rm \mathbf R_n}
=\frac{a}{\sqrt n}\left(\begin{array}{cc}
\mathbf J-\mathbf I&0\\
0&\mathbf J-\mathbf I\end{array}
\right)
\end{align*}
\begin{align*}
{\rm \mathbf M_n}
=\frac{1}{\sqrt n}\left(ib\left(\begin{array}{cc}
\mathbf A_n&0\\
0&-\mathbf A_n\end{array}
\right)+c\left(\begin{array}{cc}
0&\mathbf A_n\\
-\mathbf A_n&0\end{array}
\right)+id\left(\begin{array}{cc}
0&\mathbf A_n\\
\mathbf A_n&0\end{array}
\right)\right),
\end{align*}
Now, we have done the preparatory work and will come to finish our proof.
We shall accomplish this by three steps.

Firstly, Suppose $b\neq0$. Define a vector $\mathbf z=\left(\mathbf u^{\prime},\mathbf v^{\prime}\right)^{\prime},\mathbf u=\left(u_1,\cdots,u_n\right)^{\prime},
 \mathbf v=\left(v_1,\cdots,v_n\right)^{\prime}$ with $$\left\{u_j,j\in\mathcal D_n\right\}=\frac{1}{\sqrt {2N}}\left\{1,e^{-i\pi {\rm sign}\left(b\right)\left(2k-1\right)/N},\cdots,e^{-i\pi{\rm sign}\left(b\right)\left(2k-1\right)\left(N-1\right)/N}\right\},$$
$$\left\{v_j,j\in\mathcal D_n\right\}=\frac{1}{\sqrt{2 N}}\left\{1,e^{i\pi{\rm sign}\left(b\right)\left(2k-1\right)/N},\cdots,e^{i \pi{\rm sign}\left(b\right)\left(2k-1\right)\left(N-1\right)/N}\right\}.$$ By Lemma \ref{le:5},

\begin{align*}\label{al:13}
&\mathbf z^*{\rm \mathbf M_n}\mathbf z\notag\\
=&\frac{1}{\sqrt n}\left(ib\mathbf z^*\left(\begin{array}{cc}
\mathbf A_n&0\\
0&-\mathbf A_n\end{array}
\right)\mathbf z+c\mathbf z^*\left(\begin{array}{cc}
0&\mathbf A_n\\
-\mathbf A_n&0\end{array}
\right)\mathbf z
+id\mathbf z^*\left(\begin{array}{cc}
0&\mathbf A_n\\
\mathbf A_n&0\end{array}
\right)\mathbf z\right)\notag\\
=&\frac{1}{\sqrt n}\bigg(\left|b\right|\cot\frac{\left(2k-1\right)\pi}{2N}+c\left(\mathbf{ u^*A_nv}-\mathbf{v^*A_nu}\right)+id\left(\mathbf{ u^*A_nv}+\mathbf{v^*A_nu}\right)\bigg).
\end{align*}
\
Note that $\overline{\mathbf u}=\mathbf v$, we obtain
\begin{equation}\label{al:14}
  \mathbf{ u^*A_nv}=\mathbf{v^*A_nu}=0.
\end{equation}
Moreover,
\begin{align*}
\mathbf z^*{\rm \mathbf R_n}\mathbf z=&\frac{a}{\sqrt n} \mathbf z^*\left(\begin{array}{cc}
\mathbf {J-I}&0\\
0&\mathbf {J-I}\end{array}
\right)\mathbf z\notag\\
=&\frac{a}{\sqrt n} \left(\mathbf{ u^*Ju}+\mathbf{v^*Jv}-1\right)\\
=&\frac{a}{2\sqrt nN} \left(\left|\sum_{j=0}^{N-1}e^{-i\pi{\rm sign}\left(b\right)\left(2k-1\right)j/N}\right|^2+\left|\sum_{j=0}^{N-1}e^{i\pi{\rm sign}\left(b\right)\left(2k-1\right)j/N}\right|^2\right)-\frac{a}{\sqrt n} \\
=&\frac{a}{2\sqrt nN} \left(\left|\frac{1-e^{-i\pi{\rm sign}\left(b\right)\left(2k-1\right)}}{1-e^{-i\pi{\rm sign}\left(b\right)\left(2k-1\right)/N}}\right|^2+\left|\frac{1-e^{i\pi{\rm sign}\left(b\right)\left(2k-1\right)}}{1-e^{i\pi{\rm sign}\left(b\right)\left(2k-1\right)/N}}\right|^2\right)-\frac{a}{\sqrt n}\\
\le&\frac{a}{\sqrt nN\sin^2{\left(\pi\left(2k-1\right)/2N\right)}} -\frac{a}{\sqrt n}.
\end{align*}
Therefore, by taking $k=[n^{1/3}]$,
\begin{align*}
s_{\max}\left(\widetilde {\widetilde {\mathbf Q}}_n\right)\ge&\mathbf{z^*\widetilde {\widetilde {\mathbf Q}}_nz}=\mathbf{z^*\left(\widetilde {\widetilde {\mathbf Q}}_n-\breve {\mathbf Q}_n\right)z}+\mathbf{z^*\left(\breve {\mathbf Q}_n-{\rm E}\breve {\mathbf Q}_n\right)z}+\mathbf{z^*({\rm E}\breve {\mathbf Q}_n)z}\\
\ge&-n^{-1/4}+s_{\min}\left(\breve {\mathbf Q}_n-{\rm E}\breve {\mathbf Q}_n\right)+\frac{\left|b\right|}{\sqrt n\tan\left(\pi\left(2k-1\right)/2N\right)}\\
-&\frac{a}{\sqrt nN\sin^2{\left(\pi\left(2k-1\right)/2N\right)}}-\frac{a}{\sqrt n}\\
\to&\infty
\end{align*}
where the last procedure follows from the fact $N/n\to1,{\rm a.s.},\lim_{x\to0}\frac{\sin x}{x}=1 \ {\rm and} \ \lim_{x\to0}\frac{\tan x}{x}=1 $. Thus, combining the arguments above, we obtain $b=0$.

Secondly, suppose $c\neq0$. Define a vector $\mathbf z=(\left(\left(1-i\right)/\sqrt2)\mathbf u^{\prime},(\left(1+i\right)/\sqrt2)\mathbf u^{\prime}\right)^{\prime},\mathbf u=\left(u_1,\cdots,u_n\right)^{\prime}$ with
$$\left\{u_j,j\in\mathcal D_n\right\}=\frac{1}{\sqrt {2N}}\left\{1,e^{-i \pi{\rm sign}\left(c\right)\left(2k-1\right)/N},\cdots,e^{-i\pi{\rm sign}\left(c\right)\left(2k-1\right)\left(N-1\right)/N}\right\}.$$
By Lemma \ref{le:5},
\begin{align}
&\mathbf z^*{\mathbf M}_n\mathbf z\notag\\
=&\frac{1}{\sqrt n}\left(ib\mathbf z^*\left(\begin{array}{cc}
\mathbf A_n&0\\
0&-\mathbf A_n\end{array}
\right)\mathbf z+c\mathbf z^*\left(\begin{array}{cc}
0&\mathbf A_n\\
-\mathbf A_n&0\end{array}
\right)\mathbf z
+id\mathbf z^*\left(\begin{array}{cc}
0&\mathbf A_n\\
\mathbf A_n&0\end{array}
\right)\mathbf z\right)\notag\\
=&\frac{1}{\sqrt n}\left|c\right|\cot\frac{\left(2k-1\right)\pi}{2N}.\notag
\end{align}
Moreover,
\begin{align*}
\mathbf z^*{\mathbf R}_n\mathbf z=&\frac{a}{\sqrt n} \mathbf z^*\left(\begin{array}{cc}
\mathbf {J-I}&0\\
0&\mathbf {J-I}\end{array}
\right)\mathbf z\notag\\
=&\frac{a}{\sqrt n} \left(2\mathbf{ u^*Ju}-1\right)\\
\le&\frac{a}{\sqrt nN\sin^2{\left(\pi\left(2k-1\right)/2N\right)}} -\frac{a}{\sqrt n}.
\end{align*}
Therefore, by taking $k=[n^{1/3}]$,
\begin{align*}
s_{\max}\left(\widetilde {\widetilde {\mathbf Q}}_n\right)\ge&\mathbf{z^*\widetilde {\widetilde {\mathbf Q}}_nz}=\mathbf{z^*\left(\widetilde {\widetilde {\mathbf Q}}_n-\breve {\mathbf Q}_n\right)z}+\mathbf{z^*\left(\breve {\mathbf Q}_n-{\rm E}\breve {\mathbf Q}_n\right)z}+\mathbf{z^*({\rm E}\breve {\mathbf Q}_n)z}\\
\ge&-n^{-1/4}+s_{\min}\left(\breve {\mathbf Q}_n-{\rm E}\breve {\mathbf Q}_n\right)+\frac{\left|c\right|}{\sqrt n\tan\left(\pi\left(2k-1\right)/2N\right)}\\
-&\frac{a}{\sqrt nN\sin^2{\left(\pi\left(2k-1\right)/2N\right)}}-\frac{a}{\sqrt n}\\
\to&\infty.
\end{align*}
Thus, we obtain $c=0$.

Finally, suppose $d\neq0$. Define a vector $\mathbf z=\left(\mathbf u^{\prime},\mathbf u^{\prime}\right)^{\prime},\mathbf u=\left(u_1,\cdots,u_n\right)^{\prime}$ with
$$\left\{u_j,j\in\mathcal D_n\right\}=\frac{1}{\sqrt {2 N}}\left\{1,e^{-i\pi{\rm sign}\left(d\right)\left(2k-1\right)/N},\cdots,e^{-i\pi{\rm sign}\left(d\right)\left(2k-1\right)\left(N-1\right)/N}\right\}.$$
Also, by Lemma \ref{le:5},
\begin{align}
&\mathbf z^*{\mathbf M}_n\mathbf z\notag\\
=&\frac{1}{\sqrt n}\left(ib\mathbf z^*\left(\begin{array}{cc}
\mathbf A_n&0\\
0&-\mathbf A_n\end{array}
\right)\mathbf z+c\mathbf z^*\left(\begin{array}{cc}
0&\mathbf A_n\\
-\mathbf A_n&0\end{array}
\right)\mathbf z
+id\mathbf z^*\left(\begin{array}{cc}
0&\mathbf A_n\\
\mathbf A_n&0\end{array}
\right)\mathbf z\right)\notag\\
=&\frac{1}{\sqrt n}\left|d\right|\cot\frac{\left(2k-1\right)\pi}{2N}.\notag
\end{align}
Moreover,
\begin{align*}
\mathbf z^*{\mathbf R}_n\mathbf z=&\frac{a}{\sqrt n} \mathbf z^*\left(\begin{array}{cc}
\mathbf {J-I}&0\\
0&\mathbf {J-I}\end{array}
\right)\mathbf z\notag
=\frac{a}{\sqrt n} \left(2\mathbf{ u^*Ju}-1\right)\\
\le&\frac{a}{\sqrt nN\sin^2{\left(\pi\left(2k-1\right)/2N\right)}} -\frac{a}{\sqrt n}.
\end{align*}
Therefore, taking $k=[n^{1/3}]$,
\begin{align*}
s_{\max}\left(\widetilde {\widetilde {\mathbf Q}}_n\right)\ge&\mathbf{z^*\widetilde {\widetilde {\mathbf Q}}_nz}=\mathbf{z^*\left(\widetilde {\widetilde {\mathbf Q}}_n-\breve {\mathbf Q}_n\right)z}+\mathbf{z^*\left(\breve {\mathbf Q}_n-{\rm E}\breve {\mathbf Q}_n\right)z}+\mathbf{z^*({\rm E}\breve {\mathbf Q}_n)z}\\
\ge&-n^{-1/4}+s_{\min}\left(\breve {\mathbf Q}_n-{\rm E}\breve {\mathbf Q}_n\right)+\frac{\left|d\right|}{\sqrt n\tan\left(\pi\left(2k-1\right)/2N\right)}\\
-&\frac{a}{2\sqrt nN\sin^4{\left(\pi\left(2k-1\right)/2N\right)}}-\frac{a}{\sqrt n}\\
\to&\infty.
\end{align*}
Thus, we get $d=0$.
\subsection{Necessity of condition (iii)}
Applying the sufficiency part, we can get condition (iii).

So far we have completed the proof of Theorem \ref{th1}.


\begin{thebibliography}{10}

\bibitem{bai2010spectral}
Z.~D. Bai and Jack~William Silverstein.
\newblock {\em Spectral analysis of large dimensional random matrices}.
\newblock Springer, 2010.

\bibitem{BaiYin1993}
Z.~D. Bai and Y.~Q. Yin.
\newblock Limit of the smallest eigenvalue of a large dimensional sample
  covariance matrix.
\newblock {\em The Annals of Probability}, 21(3):pp. 1275--1294, 1993.

\bibitem{bai1988necessary}
Z.~D. Bai and Y.~Q. Yin.
\newblock Necessary and sufficient conditions for almost sure convergence of
  the largest eigenvalue of a wigner matrix.
\newblock {\em The Annals of Probability}, 16(4):1729--1741, 1988.

\bibitem{chevalley1946lie}
C~Chevalley.
\newblock Lie groups.
\newblock {\em Princeton UP}, 1946.

\bibitem{PhysRevE.77.041108}
David~S. Dean and Satya~N. Majumdar.
\newblock Extreme value statistics of eigenvalues of gaussian random matrices.
\newblock {\em Phys. Rev. E}, 77:041108, Apr 2008.

\bibitem{1973}
C.~A. Deavours.
\newblock The quaternion calculus.
\newblock {\em The American Mathematical Monthly}, 80(9):pp. 995--1008, 1973.

\bibitem{Dumitriu2008}
I.~Dumitriu and P.~Koev.
\newblock Distributions of the extreme eigenvaluesof beta–jacobi random
  matrices.
\newblock {\em SIAM Journal on Matrix Analysis and Applications}, 30(1):1--6,
  2008.

\bibitem{dyson1962threefold}
Freeman~J Dyson.
\newblock The threefold way. algebraic structure of symmetry groups and
  ensembles in quantum mechanics.
\newblock {\em Journal of Mathematical Physics}, 3(1199), 1962.

\bibitem{Furedi1981}
Z.~F\"{u}redi and J.~Koml\'{o}s.
\newblock The eigenvalues of random symmetric matrices.
\newblock {\em Combinatorica}, 1(3):233--241, 1981.

\bibitem{ginibre:440}
Jean Ginibre.
\newblock Statistical ensembles of complex, quaternion, and real matrices.
\newblock {\em Journal of Mathematical Physics}, 6(3):440--449, 1965.

\bibitem{hamilton1866elements}
William~Rowan Hamilton and William~Edwin Hamilton.
\newblock {\em Elements of quaternions}.
\newblock London: Longmans, Green, \& Company, 1866.

\bibitem{juhasz1978spectrum}
Ferenc Juh\'{a}sz.
\newblock On the spectrum of a random graph.
\newblock 25:313--316, 1978.

\bibitem{mehta2004random}
Madan~Lal Mehta.
\newblock {\em Random matrices}, volume 142.
\newblock Access Online via Elsevier, 2004.

\bibitem{o2013universality}
Sean O'Rourke and Van Vu.
\newblock Universality of local eigenvalue statistics in random matrices with
  external source.
\newblock {\em arXiv preprint arXiv:1308.1057}, 2013.

\bibitem{tao2011wigner}
Terence Tao and Van Vu.
\newblock The wigner-dyson-mehta bulk universality conjecture for wigner
  matrices.
\newblock {\em Electronic J. Probab}, 16:2104--2121, 2011.

\bibitem{tracy1994level}
Craig~A Tracy and Harold Widom.
\newblock Level-spacing distributions and the airy kernel.
\newblock {\em Communications in Mathematical Physics}, 159(1):151--174, 1994.

\bibitem{Wang2012}
Dong Wang.
\newblock The largest eigenvalue of real symmetric, hermitian and hermitian
  self-dual random matrix models with rank one external source, part i.
\newblock {\em Journal of Statistical Physics}, 146(4):719--761, 2012.

\bibitem{wigner1955}
Eugene~P. Wigner.
\newblock Characteristic vectors of bordered matrices with infinite dimensions.
\newblock {\em Annals of Mathematics}, 62(3):pp. 548--564, 1955.

\bibitem{wigner1957}
Eugene~P. Wigner.
\newblock Characteristics vectors of bordered matrices with infinite dimensions
  ii.
\newblock {\em Annals of Mathematics}, 65(2):pp. 203--207, 1957.

\bibitem{wigner1958}
Eugene~P. Wigner.
\newblock On the distribution of the roots of certain symmetric matrices.
\newblock {\em Annals of Mathematics}, 67(2):pp. 325--327, 1958.

\bibitem{wigner1951statistical}
Eugene~P. Wigner and PAM Dirac.
\newblock On the statistical distribution of the widths and spacings of nuclear
  resonance levels.
\newblock 47:790--798, 1951.

\bibitem{yinbai2013}
Yin, Y.Q. Bai, Z.~D. and Hu, J.
\newblock On the semicircular law of large dimensional random quaternion
  matrices.
\newblock {\em arXiv preprint arXiv:1309.6937}, 2013.

\bibitem{zhang1997quaternions}
Fuzhen Zhang.
\newblock Quaternions and matrices of quaternions.
\newblock {\em Linear algebra and its applications}, 251:21--57, 1997.

\end{thebibliography}

\end{document}